\documentclass[letterpaper, 10 pt, conference]{ieeeconf}  % Comment this line out if you need a4paper

\IEEEoverridecommandlockouts                              % This command is only needed if 
\usepackage{amsmath,bm} % assumes amsmath package installed
\usepackage{amssymb}  % assumes amsmath package installed
\usepackage{amsfonts}

\usepackage{amsthm}

\usepackage{graphicx} 
\usepackage{algorithm}
\usepackage{algpseudocode}
\usepackage{etoolbox}
\usepackage{units}
\usepackage[hidelinks]{hyperref}

\usepackage{cite}
\usepackage{lpic}
\usepackage{overpic}
\usepackage{mathrsfs}
\usepackage{todonotes}

\usepackage{bbm}

\usepackage{xcolor}

\allowdisplaybreaks[4]

\DeclareFontFamily{U}{mathx}{\hyphenchar\font45}
\DeclareFontShape{U}{mathx}{m}{n}{
      <5> <6> <7> <8> <9> <10> gen * mathx
      <10.95> mathx10 <12> <14.4> <17.28> <20.74> <24.88> mathx12
      }{}
\DeclareSymbolFont{mathx}{U}{mathx}{m}{n}
\DeclareFontSubstitution{U}{mathx}{m}{n}
\DeclareMathSymbol{\temp}{\mathbin}{mathx}{'341}

\newtheorem{thm}{Theorem}[section]

\newtheorem{prop}[thm]{Proposition}
\newtheorem{asum}[thm]{Assumption}

\newtheorem{defn}{Definition}[section]

\newtheorem{exmp}{Example}[section]
\newtheorem{rem}{Remark}
\providecommand{\span}{\text{span}}

\providecommand{\II}{\mathbb I} 
\providecommand{\PP}{\mathbb P}
\providecommand{\RR}{\mathbb R}

\providecommand{\PP}{\mathbb P}
\providecommand{\RR}{\mathbb R}

\newcommand{\mc}[1]{\mathcal{#1}}
\providecommand{\E}{\mathcal E}
\providecommand{\A}{\mathcal A}
\providecommand{\B}{\mathcal B}
\providecommand{\C}{\mathcal C}
\providecommand{\D}{\mathcal D}
\providecommand{\G}{\mathcal G}
\providecommand{\Hset}{\mathcal H}
\providecommand{\I}{\mathcal I}
\providecommand{\M}{\mathcal M}

\providecommand{\Sset}{\mathcal S}
\providecommand{\Pset}{\mathcal P}
\providecommand{\W}{\mathcal W} %\mathfrak{RQSZ}

\providecommand{\X}{\mathcal X}

\providecommand{\Z}{\mathcal Z}
\providecommand{\U}{\mathcal U}
\newcommand{\Mn}[1]{\begin{pmatrix}#1\end{pmatrix}} %normale Klammern, normal bracket
\newcommand{\Mo}[1]{\begin{matrix}#1\end{matrix}} %ohne Klammern, no brackets
\title{\LARGE \bf
Tube-based Robust Model Predictive Control for a Distributed Parameter System Modeled as a Polytopic LPV \\(extended version)}

\author{Joe Ismail\textsuperscript{1, \dag} and Steven Liu\textsuperscript{1}%
\thanks{\textsuperscript{1} Department of Electrical and Computer Engineering, Technical University of Kaiserslautern, Kaiserslautern 67663, Germany.}%
\thanks{\textsuperscript{\dag} Correspondence e-mail,~\texttt{ismail@eit.uni-kl.de}}}
\begin{document}
\maketitle
\thispagestyle{empty}
\pagestyle{empty}

%%%%%%%%%%%%%%%%%%%%%%%%%%%%%%%%%%%%%%%%%%%%%%%%%%%%%%%%%%%%%%%%%%%%%%%%%%%%%%%%
\begin{abstract}
Distributed parameter systems (DPS) are formulated as partial differential equations (PDE). Especially, under time-varying boundary conditions, PDE introduce force coupling. In the case of the flexible stacker crane (STC), nonlinear coupling is introduced. Accordingly, online trajectory planning and tracking can be addressed using a nonlinear model predictive control (NMPC). However, due to the high computational demands of a NMPC, this paper discusses a possibility of embedding nonlinearities inside a linear parameter varying (LPV) system and thus make a use of a numerically low-demanding linear MPC. The resulting mismatches are treated as parametric and additive uncertainties in the context of robust tube-based MPC (TMPC). For the proposed approach, most of the computations are carried out offline. Only a simple convex quadratic program (QP) is conducted online. Additionally a soft-constrained extension was briefly proposed. Simulation results are used to illustrate the good performance, closed-loop stability and recursive feasibility of the proposed approach despite uncertainties. 
\end{abstract}

%\begin{keywords}
%stacker cranes, model predictive control, Tube-based robust model predictive control, linear parameter-varying
%\vskip-\baselineskip
%\end{keywords}

\section{INTRODUCTION}
Stacker cranes (STC) are widely used in container terminals and large automated warehouses to execute fast and accurate positioning maneuvers for transporting purposes. Due to the flexible and slender form of the mast, undesired vibrations arise. These vibrations increase the material fatigue and reduce the productivity as well as the positioning accuracy. As a counteraction to avert such problems, many control strategies have been presented previously. 
Due to the flexible structure and moving load, STC are distributed parameter systems (DPS), which belong to the class of partial differential equations (PDE) with time-varying boundary conditions. These conditions raise nonlinearities in the lumping approaches needed to obtain ordinary differential equations (ODE) \cite{Is19Model}. For a wide operating range of the STC, model nonlinearities become very distinct, which must be explicitly addressed. 
The control techniques studied so far can be classified as open or closed-loop control techniques for either feedforward or feedback. 
Mostly the aim of open-loop techniques is the proper shaping of the reference trajectories to greatly reduce the motion-induced vibrations, as the widely used input-shaping approach, e.g. \cite{zhang2015dynamic}. Open-loop techniques include also flatness-based motion planning or inversion-based control approaches, like in \cite{zimmert2010active}, which are widely used as well. In contrast, closed-loop techniques make mostly a use of the passive nature of cranes and deploy energy-shaping strategies to decrease the actual energy and thus the vibrations, like in \cite{Staudecker08}. A more advanced control technique, is the model predictive control (MPC), which is however often used as a feedback control, as in \cite{galkina2018flatness}. In this, MPC is applied for tracking an offline generated flatness-based reference trajectories. However, MPC can combine both open and closed-loop techniques. This is advantageous, particularly in view of the steadily changing open-loop vibrational behavior of the STC. Due to the variable lift position, the open-loop frequency response changes steadily. This is a key difference of the model used here compared to the widely used modeling approaches of the STC, where the vibrations are modeled with a fixed open-loop frequency response \cite{Is19Model}. Based on this model, \cite{Is19Control} introduced an online feedforward and feedback scheme for active vibration damping control using a soft-constrained nonlinear MPC (NMPC), which is technically challenging due to the computational demand.  
To mitigate the computational demand, this paper discusses a possibility of representing nonlinearities by an LPV embedding, and thus treating the resulting mismatches (i.e. lumping, linearization) as uncertainties. Hence, nominal trajectories are computed online based on a simplified nominal linear model using a simple convex quadratic program (QP). An offline designed robust tube enforces the actual trajectories to track the nominal ones and at the same time guarantee the constraint satisfaction, asymptotic stability and recursive feasibility. The performance is benchmarked with a nonlinear MPC. TMPC is widely studied, e.g. in \cite{langson2004robust, rawlings2009model, kouvaritakis2016model}. Also the idea of embedding nonlinearities inside an LPV system was presented before in \cite{chisci2003gain, cisneros2016efficient}. \cite{chisci2003gain} uses a discrete scheduling set, resulting a collection of uncertain linear models. In \cite{cisneros2016efficient}, knowledge of the scheduling map was exploited already in the prediction stage. Based on an initial guess of the future scheduling trajectory, a simple linear time-varying (LTV) MPC problem was solved. \cite{hanema2017stabilizing} developed a nonlinear MPC based on LPV embeddings, which integrates the explicit use of a scheduling map from \cite{cisneros2016efficient} into a tube-based LPV MPC formulation, which is very close to our scheme. However, the tube parametrization considered in \cite{hanema2017stabilizing} is a so-called homothetic tube, which is different from our parametrization. A sequence of tubes are computed and then one tube and a terminal set are parametrized, which is in fact slightly different from other approaches. Additionally, the soft-constrained MPC approach is considered, which is also different.           
%
% 
%
%
%
%
%
%-----------------------------------------------------------------------------
%-----------------------------------------------------------------------------
%
%
\\
\\
\textit{Contribution:}
The main contribution of this paper is the introduction of a simple online feedforward and feedback scheme for active vibration damping control using a TMPC. This scheme reduces the online computations to a convex quadratic program (QP) with a sparse structure. At the same time TMPC deals with nonlinearities and model lumping mismatches. As a result, actual trajectories are forced to track the nominal trajectories. To penalize any close operation beside the resonance frequencies, resonance frequencies are considered in the frame of soft-constrained MPC. Additionally, this paper gives a brief discussion on stability and feasibility.  
%the relation between TMPC and the classical control in generalized form to contains dissipative systems. 
%
Treating STC as an LPV model is not new. \cite{hajdu2016mast} already used LPV modeling, which is however different from ours due the used modeling technique. Additionally, the obtained system was used for closed-loop control. In contrast, combining both open and closed-loop based on an LPV system is new and dealing with the accompanied uncertainties using TMPC is also new according to the best knowledge of the authors. Additionally, to our surprise the parametrization of tube cross-section turned out to be new in context of TMPC. 
%
%-----------------------------------------------------------------------------
%-----------------------------------------------------------------------------
%
\\ 
This paper is structured as followed: Section \ref{sec:Preliminaries} gives an introduction to the dynamic formulation of a STC. Section \ref{sec:ProblemFormulation} introduces the problem formulation using an LPV formulation and also introduces the formulation of the TMPC problem. Section \ref{sec:Analysis} discusses some theoretical properties. Section \ref{sec:MPCwithSTC} addresses the targets of the paper. Section \ref{sec:result} shows the simulation results. Finally, Section \ref{sec:conc} provide a conclusion.
%http://cse.lab.imtlucca.it/~bemporad/teaching/ac/pdf/04a-TD_sys.pdf
%

\section{Preliminaries}
\label{sec:Preliminaries}
\subsection{Mathematical Notations}
\label{subsec:notations}
Let $\II$ denote the set of non-negative integers. For $n, m \in \II \cup {\infty}$, let $\II_{\geq 0}$ and $\II_{n:m}$ denotes the sets $\{ r \in \II: n \leq r \leq m\}$, respectively. Similarly, $\RR_{\geq 0}$ denote the non-negative real numbers, $\RR^{n}$ are a real valued $n$-vectores, $\PP_N(x, r)$ is an MPC optimization problem with horizon $N$, initial state $x$ and reference $r$. $||.||$ and $\| . \|_\infty$ are the $l_2$ and the $l_\infty$ norms. A set $\A \subset \RR^n$ is called a proper C-set, or PC-set, if it is convex, compact (i.e., closed and bounded), and has a non-empty interior that includes the origin. The convex hull of a set $\A \subset \RR^n$ is denoted by $Co\{\A\}$. A convex $\Hset$-polytope denotes a bounded intersection of $q$ closed half-spaces $\Pset=\{x \in \RR^n: Cx\leq d, C\in \RR^{q \times n}, d \in \RR^q\}$. For sets $\A, \B$ and a scalar $\alpha \in \RR$ let $\alpha\A=\{\alpha \A:\alpha\in\A\}$. The Minkowski sum of two sets is defined as~$\A\oplus\B=\{z:z=a+b,a\in\A,~b\in\B\}$ and the Pontryagin difference between two sets as~$\A\ominus\B=\{z:z+b\in\A ,~\forall b\in\B\}$. 
\subsection{Nonlinear Dynamic}
\label{subsec:DynamicSetup}
Fig. \ref{fig:CTS} depicts the model of the STC. Both the carriage and the lift positions are controllable in terms of actuation. Carriage, lift and the tip mass are modeled as lumped masses $m_c$, $m_l$ and $m_t$ respectively. $(x_t,L), (x_l,y_l)$ and $(x_c,0)$ denote the coordinate of the tip, lift and carriage respectively. $EI$ represents the flexural rigidity. The cross-sectional area of the beam is denoted as $A$ and density as $\rho$. The actuator forces are axial forces and denoted as $F_1$ and $F_2$.
%
%\vspace{-7mm}
\begin{figure}[htbp]
\centering
\includegraphics[scale=0.48]{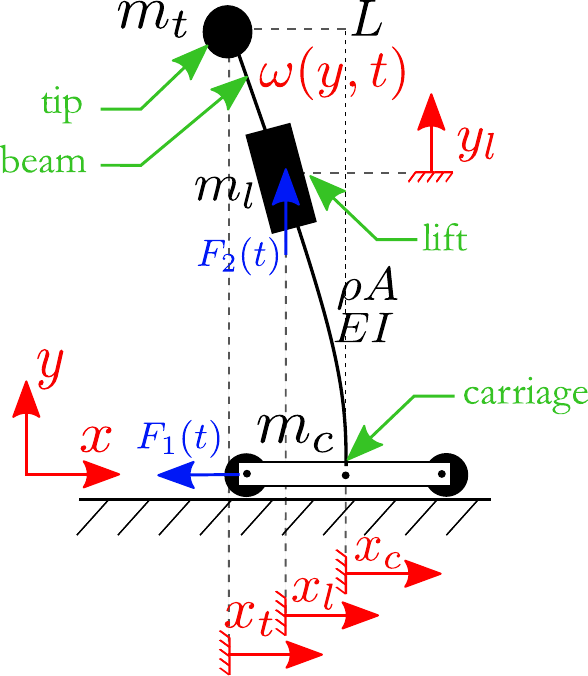}
\vspace{-3mm}
\caption {flexible stacker crane}
\label{fig:CTS}
\end{figure}
For the beam of length $L$, the spatial variable is given as $y \in \mathbb R_{0:L}$. $\omega(y,t)=\omega_y$ denote the absolute axial deflection with time-dependency $t$.
Moving lift position and load changes induce a variable open-loop frequency response. To capture these changes a model of two Euler-Bernoulli beams with time-varying boundary conditions is considered in \cite{Is19Model} and has been adapted here. In the lumping approach, two modes have been considered in a trade-off between fidelity (approximation accuracy) and computational cost. Therefore, the axial deflection $\omega(y,t)$ is thus assumed to have two degree of freedom denoted as $\phi(t) \in \mathbb R^{2}$. Together with the axial displacements, the resulting system is a fourth-order system with the following generalized coordinates,
\begin{align}
q(t)=\begin{pmatrix}x_c(t) & y_l(t) & \phi_1(t)  & \phi_2(t)\end{pmatrix}^T. \label{eq:states}
\end{align} 
Accordingly, the equations of motion is rearranged as $\ddot{q} = f(q,\dot{q},F_{ext})$, then rewritten in a lighter and familiar notation,
\begin{align}
\label{eq:ODEMS}
\dot{x} &= f(x,u),~~\text{with}~~x(0)=x_0,
\end{align}
where $f:\mathbb R^n\times \mathbb R^m \mapsto \mathbb R^n$, such that $n, m$ are the state and input order respectively. 
Details are given in \cite{Is19Model}, which are shortened here due to the marginal contribution to the sequel.
\subsection{Polytopic LPV Representation}
\label{subsec:LPVSetup}
The main idea of the LPV technique is to embed the system nonlinearities given in \eqref{eq:ODEMS} in time-varying parameters inside an LPV representation. As a result, the nonlinear dynamic is replaced by an affine composition of linear systems. In an LPV system, the dynamical mapping between inputs and outputs is linear, while the mapping itself depends on a time-varying scheduling parameter $p$. A commonly used method to obtain an LPV description of a nonlinear model \eqref{eq:ODEMS} is the classical Jacobian linearization around a series of equilibrium points or reference trajectory $(x_r, u_r)$ as,
\begin{align}
	\dot \delta_x = \underbrace{\frac{\partial f~}{\partial x^T}\bigg\rvert_{x_r,u_r}}_A \delta_x + \underbrace{\frac{\partial f~}{\partial u^T}\bigg\rvert_{x_r,u_r}}_B \delta_u + \underbrace{f(x_r,u_r)-\dot x_r}_{R(x,u)} \label{eq:linearz}
\end{align}
with $\delta_x=x-x_r,~\delta_u=u-u_r$. The affine term $R$ is non-zero and gives the linearization error. The rearranged model in \eqref{eq:ODEMS} contains state $x \in \RR^n$, and input $u \in \RR^m$ defined as, $x =\Mn{ x_c & \dot x_c & y_l & \dot y_l & \phi_1 & \dot \phi_1 & \phi_2 & \dot \phi_2}^T \in \RR^8$, $u=\Mn{F_1 & F_2}^T \in \RR^2$. The reference trajectories are stabilizing and feasible trajectories of the STC, which are computed offline as, $x_r=\{x: x=\dot x=\{0\} \setminus y_l \neq \{0\} \}$, $u_e=\{u: u_1=\{0\}, u_2 \neq \{0\} \}$. 
With increasing lift- height $y_l$ and mass $m_l$ nonlinearities becomes mores distinct \cite{Is19Model}. To capture these nonlinearities, feasible reference trajectories are obtained along the whole range of lift- height and mass changes, both within the given boundaries. This induce an LPV system description, where the time-variant elements are considered as a parametric uncertainty. 
$\breve{x}=[m_l, y_l]^T \in \RR^d$ denote the parametric uncertainty, which is known and consist of both a parameter and a state, with $d$ as the order of the parametric uncertainty. 
Both are considered to be limited over a $\Hset$-polytopic set,  
\begin{align}
\breve{\X} = \{\breve{x} \in \RR^{d}: \A_m m_l + \A_y y_l \leq b_l \},
\end{align}
with $\breve{x} \in \breve{\X} \subseteq \mathbb R^d$ to be compact, $\A_m =(1~~0~-1~~0)^T, \A_y =(0~~1~~0~-1)^T$ and upper and lower boundary vector $b_l=(\overline{\breve{x}}~~\underline{\breve{x}})^T$. A scheduling parameter is now defined as $p$, which is state-dependent through the scheduling map $T:\breve{\X} \rightarrow \Pset$. The set $\Pset$ is defined as the scheduling set.
\begin{defn}
Due to the state-dependency of scheduling parameter $p$, the resulting LPV system is regarded to be a quasi-LPV. Additionally, an LPV system is polytopic, when matrices can be represented as $A(p_k), B(p_k)$, with an affine dependency to $p_k$, which varies on a known scheduling set $\Pset$. \label{asum:QLPV} 
\end{defn}
\noindent Additionally, increasing lift height $y_l$ and mass $m_l$ induce stronger linearization errors, which are however limited due to the bounded polytope $\breve{\X}$. For this, the linearization errors given in \eqref{eq:linearz} as $R(x,u)$, is considered here as an additive uncertainty $w_k$. For discretization, the Euler method was used, with a sampling time $T_s$. Since both $y_l$ and $m_l$ are known, i.e. measurable, a discrete polytopic quasi-LPV model can be obtained, with $x_{k+1} = A(p_k)~x_k + B(p_k)~u_k +  w_k$ and state-dependent scheduling map $p_k = T(x_k)$, written as,
 %is replaced here by a lighter notation. Usually the scheduling map contains the nonlinearities in terms of a scheduling variable, which has an affine dependency to system matrices. In sake of simplifying the notation, the scheduling map is still inside the matrices given as,
%
%
%
%
\begin{align}
	x_{k+1} &= A_{p_k}~x_k + B_{p_k}~u_k +  w_k,~x(0)=x_0, \label{eq:LPVsys}\\
	p_k &= T(x_k),  \label{eq:schedulingMap}
\end{align}
with, 
\vspace{-6mm}
\begin{align*}
A_{p_k}=
\sbox0{$
%\id_{(2\times2)} \otimes \Mo{1 & T_s \\ 0 & 1}
\textbf{I}_{(2\times2)} \otimes \Mo{1 & T_s \\ 0 & 1}
$}
\sbox1{$
\Mo{ p_{2(4\times4)} } %ohne Klammern, no brackets
$}
\Mn{ 
\begin{array}{c|c}
\usebox{0}&\begin{array}{c} p_{1(4\times4)}  \\  \textbf{0}_{(2\times4)} \end{array}\\
\hline
\textbf{0}_{(4\times4)} &\usebox{1}
\end{array}
},
\end{align*}
\begin{align*}
B_{p_k}=
\Mn{ 
	    p_{3(2\times1)} & \textbf{0}_{(2\times1)} \\
			\textbf{0}_{(2\times1)} & p_{4(2\times4)} \\
			p_{5(4\times1)} & \textbf{0}_{(4\times1)} 
} 
,~~~
u=
\Mn{F_1 \\ F_2}. 
\end{align*}
The scheduling map $T$ consists of a set of nonlinear terms given as, $ T(x_k)= \{\Xi_{2\times4},  \Gamma_{4\times4},  \Pi_{2\times 1}, \Upsilon_{2\times 1}, \Lambda_{4\times 1}\}$. 
A series of scheduling parameters $p_k=\{p_i\}_i^5$ are given here, which are affine to $A_{p_k}$ and $B_{p_k}$. Matrix $A_{p_k}:\RR^{d}\rightarrow \RR^{n\times n}$, $B_{p_k}:\RR^{d}\rightarrow  \RR^{n \times m}$ are known to lie in the scheduling set $\Pset$. $k \in \II_{\geq 0}$ denote the discrete time index, $x_k:\II\rightarrow\X\subseteq\RR^n$ is the state variable, $u_k:\II\rightarrow\U\subseteq\RR^m$ is the control input, $p_k:\II\rightarrow \Pset \subseteq\RR^d$ is the scheduling variable and $w_k:\II \rightarrow\W\subseteq\RR^n$ is the additive uncertainty. $\X,~\U$, $\Pset$ and $\W$ are state-, input-, scheduling and uncertainty sets, respectively, which are assumed to be PC-polytopic sets and known exactly. The matrices $A_{p_k}, B_{p_k}$ are considered as,
\begin{align}
\Mn{A_{p_k} \\ B_{p_k} }
\in  
\underbrace{Co 
\bigg\{\ \hspace{-1mm}
\Mn{ A_j \\ B_j }\bigg\}}_{\mathcal D}
:= \sum_{j=1}^{M} \sigma^j(p_k) 
\Mn{A_j \\ B_j}. \label{eq:affinefunc}
\end{align}
\begin{asum}
The matrix pairs of vertices $(A_j,~B_j)\in \RR^{n\times n} \times \RR^{n\times m},~j \in \II_{1:M}$ are stabilizable.
\label{asum:stabilize}
\end{asum}
\noindent with $j \in \II_{1:M}$, denote $0 \leq \sigma^j(p_k) \leq 1$ as the weight of each pairs of vertices and let $\sum_{j=1}^{M} \sigma^j(p_k)=1$. The vector of parameters evolves inside a polytope represented by the convex hull $\mathcal D$ with $M$ as the number of vertices. The disturbance constraint set is represented by $\W=Co\{w_j, j \in \II_{1:M}\}$. In \eqref{eq:affinefunc}, it is obvious that, the system matrices are real affine function of $p_k$. Together with $p_k$ as in Definition \ref{asum:QLPV}, such an LPV system is referred as polytopic quasi-LPV system. The nonlinear elements in \eqref{eq:schedulingMap} are not listed here due to the shortage of space. However more details with the related derivations are given as a MuPAD Notebook online on \cite{Exchange19}. 
\begin{rem}
Recall \eqref{eq:linearz} and \eqref{eq:LPVsys}, the additive disturbances acting on the linearised model represents the linearization error, can be obtained using the Mean Value Theorem.
\end{rem}

\section{Problem Formulation}
\label{sec:ProblemFormulation}
Now consider the uncertain discrete-time LPV system, given in \eqref{eq:LPVsys} and rewrite it as,
\begin{align}
	x_{k+1} = (A_0+\Delta_A)~x_k + (B_0+\Delta_B)~u_k +  w_k,
	\label{eq:LPVsysNeu}
\end{align}
with $x(0)=x_0$. Further, $A_0$ and $B_0$ denote the nominal system matrices, by averaging the parametric uncertainties. Additionally, $\Delta_A$ and $\Delta_B$ denote the worst case uncertainities ($p=\overline{p}$), such that $(\Delta_A,~\Delta_B) \in \mathcal D$. In this paper more general hard constraints, $\M \subseteq \X \times \U$, on the state-input space given in polytopic representation are imposed,
\begin{align}
(x_k,u_k) \in \M:= \C_x x_k + \D_u u_k \leq \E ,~k \in \II_{0:N-1}, \label{eq:MixedConstraints}
\end{align}
with $\C_x \in \RR^{n_e\times n}, \D_u \in \RR^{n_e\times m}$ and $\E\in \RR^{n_e}$ to be the matrix element of the polytope in state-input space in $\Hset$-polytope representation, i.e. $(x_k,~u_k)\in(\X,~\U)$. $N$ is a given integer i.e. horizon. Additionally, a polytopic terminal constraint set is introduced as,
\begin{align}
\X_f:= \G_x x_k \leq \mc F,~k \in \II_{N:\infty}, \label{eq:TerminalConstraints}
\end{align}
with $\G_x \in \RR^{n_f\times n}, \mc F \in \RR^{n_f}$, which specifies the hyperplanes bounding of $x_N \in \X_f \subset \X$.
\subsection{State Decomposition}
\label{subsec:StateDecomposition}
The uncertain system is firstly decomposed into a nominal and the error of the uncertain system, denoted as $z$ and $e=x-z$, respectively,
\begin{align}
	z_{k+1} &= A_0~z_k + B_0~v_k, \label{eq:decompNom} \\
	%e_{k+1} = A_0~r_k + w_k,	\label{eq:LPVsysNeu}
	e_{k+1} &= A_c~e_k + A_d~e_k+ \Delta_A z_k + \Delta_B v_k + w_k, \label{eq:decompError}
\end{align}
with $A_c=A_0+B_0~K$ as the closed-loop system matrix, which is assumed to be Hurwitz matrix due to Assumption \ref{asum:stabilize} and $A_d=\Delta_A+\Delta_B~K$ as a uncertain system matrix. $K$ is a pre-stabilizing state-feedback designed for the nominal system \eqref{eq:decompNom} and $v_k$ denote the nominal input. Since prediction depend on the input, which contains the uncertain state, the pre-stabilizing $K$ is necessary to obtain stabilized predictions and avoid infeasibility.
\subsection{Disturbance Invariant Tube}
\label{subsec:Invariantmpc}
Due to the state- and input- dependency of the error $e$ in \eqref{eq:decompError}, it is difficult to obtain an analytic solution of the difference equation \eqref{eq:decompError}, which can be computed offline. Such a solution written as a set-valued function describes the disturbance set. \cite{rawlings2009model} proposed an additional term to the nominal system dynamic to eliminate $A_d~e_k+ \Delta_A z_k + \Delta_B v_k$ as $w_e=\Delta_A~x_k+ \Delta_B~u_k$. Since $(\Delta_A, \Delta_B)$ are known and present the worst case uncertainties, conservatism is added with the new additive term. $w_e$ lies in the set $\W_e$ defined as, $\W_e:=\Delta_A~\X \oplus \Delta_B K~\U$. Due to the affine mapping with $\X$ and $\U$ the tube cross-section increases with increasing admissible sets, which make this approach very conservative. Another way, was introduced by \cite{Fleming13}, which computes online the parametrization of tube cross-sections. As a result, the invariant tube cross-section become time-varying. However, for this the optimization deals with an additional decision variable. As a trade-off between conservatism \cite{rawlings2009model} and computational effort \cite{Fleming13} a simple method is introduced here, which is computed offline.     
\begin{defn}
(RPI set): A set $\Omega \subset \RR^n$ is robust positive invariant set, if $Ax+w \in \Omega$ for all $(x,w)\in (\X,\W)$ if and only of $A\Omega\oplus\W\subseteq\Omega$.
\end{defn}
\begin{asum} \label{asum:LPVToLTI}
Over the matrix pairs $(A_j, B_j), j \in \II_{1:M}$ a known and constant scheduling parameter is deployed, which inherent the multiplicative terms in each pairs. 
\end{asum}
\noindent Due to assumption \ref{asum:LPVToLTI}, \eqref{eq:LPVsysNeu} is considered to be a set of discrete linear models of matrix pairs $(A_j,~B_j)$, however with only implicit multiplicative terms. As a result, the following applies $e_{k+1} = A_c e_k + w_k$. For each pair $(A_j, B_j)$ a stabilizing feedback $K_j, j \in \II_{1:M} $ is designed. Each $K_j$ is the optimal controller gain of the discrete Ricatti equation of each matrix pair. For each $K$ a resulting uncertainty set, given as $\Sset_k^i$ is defined $\forall i \in \II_{\geq 0}: i\rightarrow \infty$ as,
\begin{align}
\Sset_k^i = \left\{\bigoplus_{j=0}^{i-1}A_c^j\mc W\right\}, 
\end{align}
with $e_k \in \Sset_k^{\infty}$, which is known to be the minimal robust positive invariant set (mRPI). However, due to computational difficulties an inner and outer approximation of a mRPI set is obtained. For this, and as its given in \cite{rakovic2005invariant}, a number of linear program problems (LP) is solved to obtain a finite integer $i$ and a scalar $\alpha \in [0,1)$, such that $A_c^i \mc W \subseteq \alpha \mc W$ and $K A_c^i \mc W \subseteq \alpha K \mc W$. As a result, for each matrix pairs an approximation of $\Sset_k^{\infty}$ is given as $\Z_{apr}(j) = (1-\alpha(i))^{-1}\Sset_k^i$ and a set of $\Z_{apr}(j)$ is obtained. For simplicity a convex hull of the union of each $\Z_{apr}(j)$ is considered here,
\begin{align}
\Z &= Co \bigg\{\bigcup_{j=1}^M \Z_{apr}(j)\bigg\}. 
\label{eq:RobustInvariantSet}
\end{align}
$\Z$ is defined as a RPI set and is PC-set. Similar computing is known for solving linear matrix inequalities for stabilizing an LPV system, but not used for TMPC according to our knowledge. Sure, it introduce also conservatism, which is however much lesser than that of \cite{rawlings2009model}. A sequence of the RPI sets centered along the nominal trajectory construct simply the cross-section of the disturbance invariant tube, denoted as set sequence $\{X_k\} := \{z_k\} \oplus \Z $ and an associated (time-varying piecewise affine) nominal policy satisfying $\{U_k\} := \{v_k\} \oplus K Z$. All possible state trajectory realizations due to uncertainties are contained in this tube $(x_k, u_k) \in (\{X_k\},\{U_k\})$.
After computing $i, \alpha$ and $\Z$, running and terminal constraints are tightened as,
\begin{subequations}
\begin{align}
\{z_k\}  \in \bar{\X} &:= \mc X \ominus \Z,& \hspace{-6mm}\forall k \in \II_{0:N-1}, \\
\{v_k\}  \in \bar{\U} &:=\mc U \ominus K \Z,& \hspace{-6mm}\forall k \in \II_{0:N-1}, \\
z_N  \in \bar{\X_f} &:= \mc X_f \ominus \Z \subseteq \bar{\X}.&
\end{align}
\end{subequations}
\subsection{Nominal MPC}
\label{subsec:nominalmpc}
Define a nominal, finite-horizon open-loop optimal problem with the typical paradigm cost function,
%
%  Constraint teighting
\begin{subequations}
\label{eq:OPC1}
\begin{align} 
\mathcal{J}^0_N \overset{\Delta}{=} \| z_N\|^2_P &~+\sum_{k=0}^{N-1} \| z_k\|^2_Q + \|  v_k\|^2_R \nonumber \\
\PP_N(z_0, .):  \mathcal{J}^*_N =&~ \underset{\bm{v}, \bm{z}}{\text{min}}~ \mathcal{J}^0_N(\bm{v},\bm{z}) \label{eq:OPC2} \\
\text{s.t.} &~~z(0)=z_0,  \label{eq:OPC3} \\
&~~z_{k+1}=A_0~z_k+B_0~v_k, \label{eq:OPC4}  \\
&~~ \bar{\M}:=\C_x x_k + \D_u u_k \leq \bar{\E}, \label{eq:OCP5}  \\
&~~ \bar{\X_f}:= \G_x x_k \leq \bar{\mc F}, \label{eq:OCP6} \\
&~~(\bm{z}, \bm{v})\overset{\Delta}{=} (z_k, z_{N}, v_k) \in \bar{\M}% \\
\end{align}
\end{subequations}
$\forall k \in \mathbb I_{0:N-1}$, where $\bm{z}:\mathbb R^n \mapsto \mathbb R^{N\times n}$ and $\bm{v}:\mathbb R^n \mapsto \mathbb R^{N\times m}$ are the nominal state and control sequences, with $\bar{\M} \subseteq \bar{\X} \times \bar{\U}$. Note that $Q\in \mathbb R^{n\times n}, Q\succeq 0 $, $R\in \mathbb R^{m\times m}$ and $R\succ 0 $ for weighting the stage costs. Additionally, assume that $P\in \mathbb R^{n\times n}, P\succeq 0 $ for weighting the terminal cost. Additionally, $\bar{\M}$ denote the tightened admissible state-input space, with $\bar{\E}$ and $\bar{\mc F}$ as the tightening boundaries in polytopic representation. For the nominal system \eqref{eq:OPC4} $K$ is the nominal feedback, which is the solution of the Riccati equation. Together with the help of the nominal state and control sequence a feedback policy i.e. closed-loop is obtained as,
\begin{align}
u_k=
\begin{cases}
K(x_k-z_k)+v_k, &\forall  k \in \mathbb I_{0:N-1}, \\
K~x_k, &\forall  k \in \mathbb I_{\geq N}.
\end{cases} \label{eq:feedbackpolicy}
\end{align}
This feedback \eqref{eq:feedbackpolicy} denote the dual-mode paradigm, i.e. Mode 1 ($\forall  k \in \mathbb I_{0:N-1}$) and Mode 2 ($\forall  k \in \mathbb I_{\geq N}$). It is worth to mention that \eqref{eq:OPC1} is written as sparse structured QP, which is an additional feature for efficient online computations.  
\subsection{Soft Constrainted MPC}
\label{subsec:Softmpc}
In \cite{Is19Control} an online feedforward and feedback scheme for active vibration damping control using a soft-constrained MPC was introduced and adapted here. For trajectory planning  and due to the variable lift position as well as the load change, different open-loop frequency responses of the beam can be obtained for each prediction. To avoid the resonance frequencies, these have been analyzed and modeled as a bounded set $\B$. For each prediction over a horizon a frequency prediction is carried out online. Any consideration of resonance frequencies as hard constraints shrinks the operational domain, especially for mechanical systems, with mostly low resonance frequencies. Therefore, this constraint is relaxed in the frame of soft-constrained MPC. Any frequency prediction obtaining a violation of $\B$ is penalized in the cost function by introduction of slack variables. As a result, feasibility is preserved and a resonance free operation can be guaranteed. Consequently, the finite-horizon open-loop optimal problem in \eqref{eq:OPC1} is extended with the inequalities,
\begin{align}
\E(x_k,u_k,\rho_k)-\B^p_{-} &\geq - s_k, \\
\E(x_k,u_k,\rho_k)-\B_{+}^p &\leq s_k,
\end{align}
where $\E(x_k,u_k,\rho_k)$ states the solution of the frequency prediction using system model. $s_k$ are the slack variables and $\B^p_{-}, \B^p_{+}$ are the lower and upper boundaries of $\B$ respectively. Additionally, required is $(s_k, s_N) \geq 0$ and the slack variables are penalized in the cost function. Due to the fact, that $\B$ is soft-constrained, and thus nominal trajectories are still feasible, no additional actions are required for constraint tightening and to guarantee the recursive feasibility.  

\section{Discussion on stability and feasibility}
\label{sec:Analysis}
Stability and feasibility for TMPC are widely studied in literature, e.g. in \cite{rawlings2009model, kouvaritakis2016model}. The key element of asymptotic stability approaches is the introduction of a terminal cost function \eqref{eq:OPC2} to guarantee the monotonic non-increasing property of the objective function in closed-loop. The terminal set is designed in \eqref{eq:OCP6} to obtain recursive feasibility. Ensuring feasibility and stability is necessary to mimic the infinite horizon MPC.
\subsection{Maximal Control Invariant Sets}
\label{subsec:CIS} 
\begin{defn} 
\label{def:CPI}
(CPI set): A set $\Omega \subset \RR^n$ is controlled positive invariant set for a given dynamics and constraints if, for all $x_k\in \Omega$, there exists input $u_k \in \RR^m$ such that constraints holds and $x_{k+1} \in \Omega$. Furthermore, $\Omega$ is maximal CPI (MCPI) set, if it is CPI and contains all other CPI sets.
\end{defn} % Cannon S.25
\noindent Feasible sets $\mc F_{N}$ are defined as a set of initial condition that steers to a terminal set over a horizon $N$. Size of $\mc F_{N}$ grow with increasing $N$, until a maximal feasible set is reached defined as $\mc F_{\infty}=\bigcup_{N=1}^\infty \mc F_{N}$. These sets are, however, not dependent on the choice of terminal set and thus are identical to the maximal CPI set or infinite time reachability set \cite{kouvaritakis2016model}. 
%
% 
%This damping arises mainly due to the mechanical-energy dissipation within the material. For this typically, the visco-elastic model is used. According the terminal set determination depends also on this damping factor.
%
%
%
%
%
%
%
%
%
%
%
%
%
%
%
\subsection{Maximal Positive Invariant Set}
\label{subsec:TerminalSet}
For Mode 2, the terminal set is computed as a maximal positive invariant (MPI) set with robust tightening to be MRPI set. Maximality is considered as the union of all RPI sets under the dynamic and constraints. For LPV systems, MRPI relay mostly on the parametric disturbed system $x_{k+1} = (A_c + A_d~K)x_k$, which induce conservatism. Here, since the time-varying scheduling parameter $p$ follows a known sequence the terminal set is computed for the last matrix pairs $(A_M, B_M)$, with $K_M$ to be the optimal controller gain of the discrete Riccati equation of the closed-loop, $x_{k+1} = (A_M+B_M K_M)x_k$. In contrast, classical TMPC use mostly only one $K$ for both nominal feedback and terminal set computation.   
%
%\begin{align}
	%x_{k+1} = (A_0+\Delta_A)~x_k + (B_0+\Delta_B)~u_k.
	%\label{eq:LPVsysNeuWAdd}
%\end{align}
%
To obtain the terminal set, the $N$-Step backward reachability is applied. Backward reachability is sometimes referred as pre-images or pre-sets. In this case the pre-sets of terminal set are computed, by defining the set of states which evolve into a target set in N-steps under given system dynamic set constraints as, 
\begin{align}
	\text{Pre}(\X_f) = \{ \G_x x_k\leq \mc F : \G_x x_{k+1}\leq \mc F \}.
	\label{eq:Pre}
\end{align}
Starting from the admissible set $\X_0=\X$ and propagating recursively as, $\X_{k}= \text{Pre}(\X_{k-1}) \cap \X,~~\forall k \in \II$.
%
%
% 
%
%   
%This defines the set of states which evolve into a target set in N-Steps for a given system dynamic and under set restrictions. In fact, its a composition of a polytope with an affine mapping. 
A backward propagating sequence $\{\X_k\}$ keeps propagation as long as, $\X_{k+1}\subseteq \X_k$ and it terminates when $\X_{k+1}=\X_k$, with $\X_k$ to be the maximal positive invariant set $\X_f\subseteq \X$. 
\subsection{Asymptotically Stability}
\label{subsec:stability} 
STC is a passive mechanical system, which provide an analogy to Lyapunov stability theory. The total mechanical energy dissipate due to material damping and the energy decline to zero. 
\begin{prop} 
\label{prop:RAS}
(RAS): The set $\Z \times \{0\}$ is a robust asymptotic stable for the decomposition $(z_{k+1} = A_0 z_k + B_0 v_k, e_{k+1} = A_c e_k + w_k)$ in the positive invariant set $\Z \times \bar{\X}$.
\end{prop}
\begin{proof} \cite{rawlings2009model}
Define $\mathcal K\!\mathcal L$ to be a comparison function with $\beta \in \mathcal K\!\mathcal L$ and $\beta(c,t)$ to be continuous and strictly increasing with increasing $c$ (decreasing with increasing $t$). Since the origin is asymptotically stable for $z_{k+1} = A_0~z_k + B_0~v_k$, there exists a $\mathcal K\!\mathcal L$ function $\beta$, such that every solution $\phi(k;z)$ of the controlled nominal system with initial state $z_0 \in \bar{\X}$ satisfies, $|\phi(k;z)|\leq\beta(|z(0)|,k), \forall  k \in \mathbb I_{\geq 0}$. Since $e(0) \in \Z$ implies $e(k)\in \Z$ for all $k \in \mathbb I_{\geq 0}$, it follows that, $|(e(k),\phi(k;z)|_{\Z \times \{0\}} \leq |e(i)|_{\Z} + |\phi(k;z)|\leq \beta(|z(0)|,k)$. Hence, the set $\Z \times \{0\}$ is a RAS in $\Z \times \bar{\X}$.     
\end{proof}
\noindent Due to the passivity of the STC, MCPI and MPI sets depend also on a stiffness-proportional damping factor of the beam, which is artificially introduced to model the actual material damping \cite{Is19Model}. For increasing damping factors $(\beta_d = 0.01, 0.5, 1)$, Fig. \ref{fig:Sets} shows increasing MPI sets (dotted) and increasing MCPI sets (dashed). For relaxed input constraints MPI takes even the size of MCPI sets. In the sequel, the damping factor is much reduced to demonstrate the convergence properties of the TMPC. Generally, for very stiff systems one might start the TMPC trajectories even from inside an invariant set. In this case, trajectory planning is not required to determine whether, but how the trajectory converges. Also, the change of lift mass induces a time-varying MCPI sets. However, since the trajectories lies inside the tube, MCPI property is already guaranteed over the given mass changes.
\vspace{-3mm}
\begin{figure}[htbp]
\begin{center}
\includegraphics[width=0.4\textwidth]{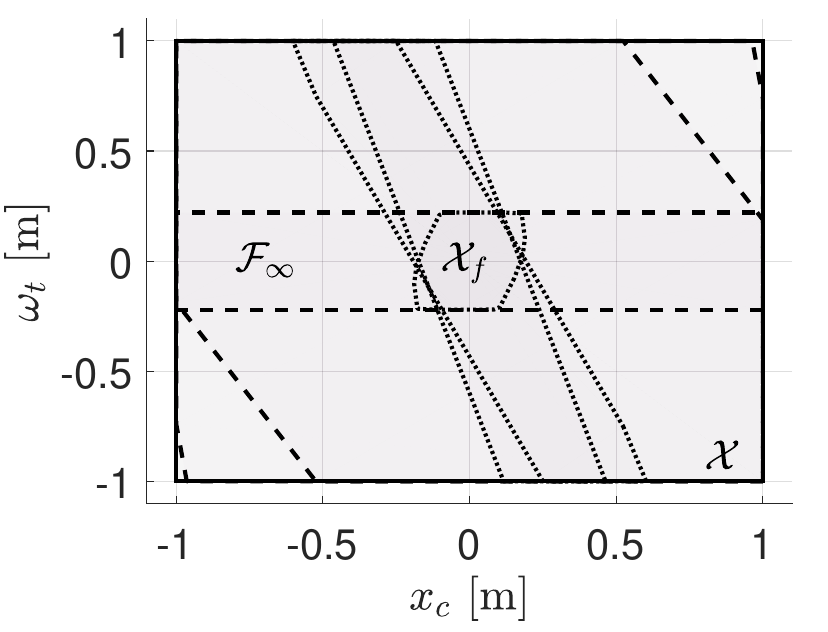}
\vspace{-3mm}
\caption{MPI sets (dotted), MCPI sets (dashed)} 
\label{fig:Sets}
\end{center}
\end{figure}

%\vspace{-5mm}
\section{Trajectory Planning and Tracking Control}
\label{sec:MPCwithSTC}
This paper deals with both, the open-loop and the closed-loop problem. For the first problem, at each sampling instant a nominal optimal control problem (OCP) is solved in context of MPC \eqref{eq:OPC1}, which aim to drive the stacker crane (STC) from a known initial state $x_0$ into the origin while minimizing an objective function. The second is a tracking problem, in particular tracking along these trajectories. The proposed TMPC deals with both problems online. The whole performance is benchmarked with a fully nonlinear MPC (NMPC).   
\vspace{-1mm}
\subsection{Nonlinear MPC (NMPC)}
\label{subsec:nmpc}
Under the same MPC conditions of \eqref{eq:OPC1}, however with the nonlinear model of \eqref{eq:ODEMS}, the NMPC problem is formulated. For this, a sequence of finite dimensional Nonlinear Programming (NLP) is solved repeatability online. NLP is obtained using the direct multiple shooting method and solved using the Sequential Quadratic Programming (SQP) method. This optimization is repeated until the terminal set is reached with stabilizing guarantee, i.e. Mode 2.
\section{NUMERICAL CASE STUDY: Stacker Crane}
\label{sec:result}
In this section, open-loop and closed-loop simulations, with the objective of active vibration damping control, are carried out for the STC. In this case study, the active vibration damping problem is stated as an OCP, which is solved in terms of TMPC as online closed-loop control. For the sake of completeness, the trajectory planning of the proposed TMPC is benchmarked with the trajectory planning due to a NMPC. For simplicity the two modes $\phi_1, \phi_2$ are merged to form the deflection $\omega_t=\sum_{j=1}^2\psi_j(y)\phi_j(t)$ as a replacing state. In this way, state dimension is reduced to $x \in \RR^6$, where the input variables stay unchanged $u \in \RR^2$. Sampling time is considered as $T_s=\unit[0.03]{s}$. Number of vertices are $M=8$. Following parameter were selected with corresponding SI units $m_c=\unit[2.888]{}, m_t=\unit[2]{}, m_l=\unit[1.5]{}, L=\unit[1.9]{}, A=\unit[3.2e^{-4}]{}, \rho=\unit[2700]{}, EI=\unit[119.4]{}$. Nominal $(A_0, B_0)$ are the middle value of the 8 vertices, which are within their polytope and multiplicatively disturbed by a random signal $\unit[-0.1651]{} \leq \|\Delta_A\|_{\infty} \leq \unit[0.1174]{}$ and $\unit[-0.0026]{} \leq \|\Delta_B\|_{\infty} \leq \unit[0.0032]{}$. The additive uncertainty is also a random signal, with $\|w_k\|_{\infty}\leq \unit[0.1]{}$. The scheduling parameter $p$ varies between $m_l \in [\unit[0.04]{kg},\unit[1.5]{kg}]$ and $y_l \in [\unit[0]{m},\unit[1.9]{m}]$. 
For $\Z_{apr}$, $\{\alpha\}_1^8=[0.0476, 0.1]$ and $i=98$. For active vibration damping, the initial deviation of the tip is considered to be $\omega_t=\unit[0.5]{m}$. Additionally, the lift has to move from an initial position $y_l=\unit[1.9]{m}$ to $\unit[0]{m}$, while carriage moves from $x_c=\unit[1]{m}$ to $\unit[0]{m}$. Moreover, the STC is subjected to constraints, such as, on the carriage position with $\|x_c\|_{\infty}\leq \unit[1]{m}$, lift position $\|y_l\|_{\infty}\leq \unit[1]{m}$. All other states are constrained to $\|.\|_{\infty}\leq \unit[4]{(m)/(m/s)}$ and control input $\|F\|_{\infty}\leq \unit[50]{N}$. Both the carriage and the lift are actuated. The weighting matrices of the stage and terminal costs of the objective function are chosen as $Q=2.5 \cdot \I\in \mathbb{R}^{6 \times 6}$ and $R=\I \in \mathbb{R}^{2 \times 2}$, with $\I$ as an identity  matrix. Slack variables weighting matrix is selected as $10^5$. $P$ is chosen to be the solution of the discrete Riccati equation of the last $\Z_{apr}(j=M)$. The optimization horizon is equal to 15. Since uncertainty exists also at the initial state, this is considered as an additional optimization variable and thus tightened. To demonstrate that, initial states are selected, which lies outside the tightened set. The initial states are then tightened by $x_0 \in \{z_0\} \oplus \Z \cap \bar{\X}$ and $x_k=\{z_k\} \oplus \Z$ holds along the horizon. A Monte-Carlo simulation of 35 simulations has been conducted. 
%OFFSET??
%
%
%
% 
% 
%
%
%
\begin{figure}[htbp]
\begin{center}
\includegraphics[width=0.5\textwidth]{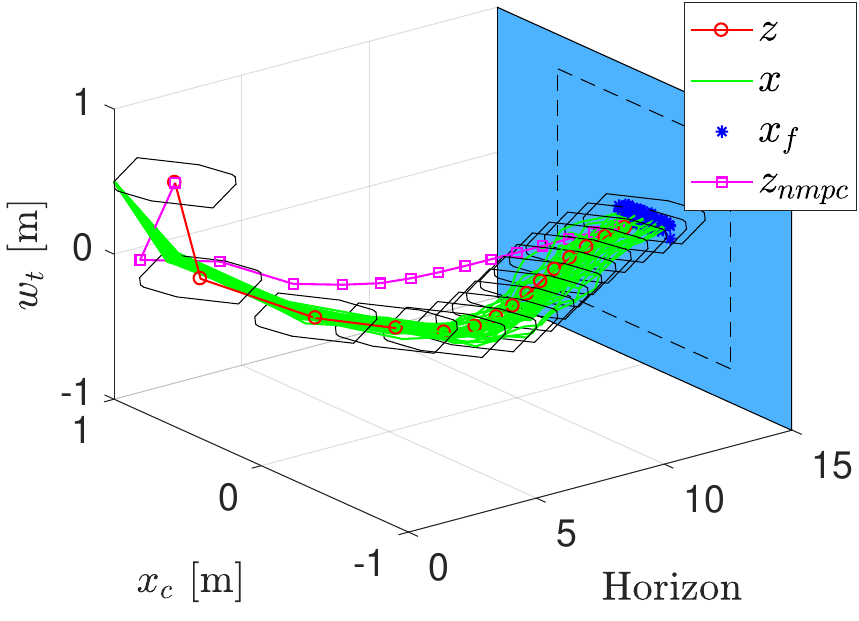}
\vspace{-3mm}
\caption{Projection of TMPC on $x_c$ and $\omega_t$ along the horizon} 
\label{fig:TMPC1}
\end{center}
\end{figure}
% 
%\vspace{-3mm}
%
\begin{figure}[htbp]
\begin{center}
\includegraphics[width=0.5\textwidth]{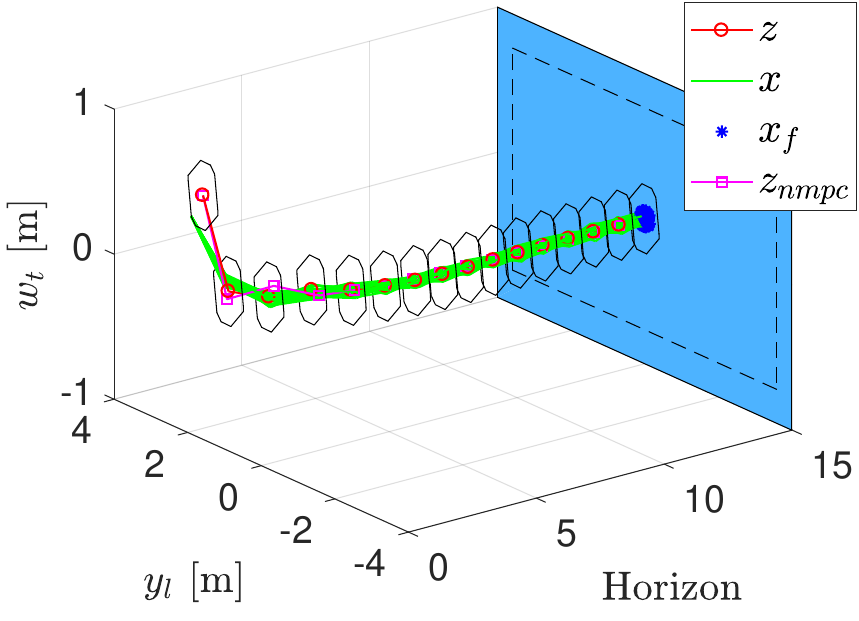}
\vspace{-3mm}
\caption{Projection of TMPC on $y_l$ and $\omega_t$ along the horizon} 
\label{fig:TMPC2}
\end{center}
\end{figure}
The simulation results of the TMPC are shown in Fig. \ref{fig:TMPC1}, \ref{fig:TMPC2}, \ref{fig:TMPC1_1} and \ref{fig:TMPC2_2}. Fig. \ref{fig:TMPC1}, Fig. \ref{fig:TMPC2} use the damping factor, known from Section \ref{subsec:stability}, as $\beta_d = 0.1$ and Fig. \ref{fig:TMPC1_1}, Fig. \ref{fig:TMPC2_2} use the damping factor $\beta_d = 0.01$ in order to simulate a behavior of the beam with high and beam with less stiffness. The six-dimensional TMPC is projected on $x_c$ and $\omega_t$ and on $y_l$ and $\omega_t$. It has been observed that, 2D plots masks sometimes underlying constraint violation, therefore 3D plot are adapted here. The nominal trajectory is given in red, as $z$. Around it, the disturbance invariant tube, consists of a set sequence of $\Z$ is shown. It can be stated, that inside this tube and despite disturbance for every actual state $x_k \in \{z_k\} \oplus \Z$ there exists an input such that $x_{k+1} \in \{z_{k+1}\} \oplus \Z$ and thus CPI. Constraints are not violated $(z \in \bar{\X}, x \in \X)$. Actual states beyond the nominal trajectories are given in asterisk blue, as $x_f$, which stays in the terminal set (Mode 2). Solid and dashed blue polytopes portrait the $\X$ and the tightened $\bar{\X}$ admissible sets. Despite the heavy tip mass $m_t$, the big lift mass $m_l$ and the position $y_l$ variations, this TMPC shows a robust performance. All 35 closed-loop simulations (actual state $x$) given in green stays within the tubes. The vertical motion given in Fig. \ref{fig:TMPC2} is less effected from the axial deviation as it is the case in Fig. \ref{fig:TMPC1}. 
Additionally, both nominal trajectories of NMPC and TMPC are compared. Especially in Fig. \ref{fig:TMPC1}, the nominal trajectory obtained from NMPC shown in magenta differs obviously from the nominal trajectory obtained from the TMPC on Fig. \ref{fig:TMPC2}. This is due to the associated nonlinearities, which became mores distinct, with the considered heavy masses. Also differences between NMPC and TMPC are obvious in the simulation of the beam with less stiffness on Fig. \ref{fig:TMPC1_1} and Fig. \ref{fig:TMPC2_2}. In this case, NMPC operates more accurate in terms of frequency prediction and thus more efficient for damping control, compared to the TMPC using the LPV representation. One of the key ideas of open-loop prediction, is the frequency prediction in order to operate away from resonance frequencies or maybe even away from other specific frequency. Due to the nonlinear model, a frequency prediction showed a high level of fidelity, at least for the first two resonance frequencies. For higher frequencies model mismatch is surely also higher. Therefore, it can be stated that, for precise frequency prediction and accurate active vibration damping control, NMPC is the better choice, however with remarkable computational effort. The proposed tube-based controller can be used to control a non-linear stacker cranes and only requires the on-line solution of a set of convex optimization problems. Obviously, this lower computational effort is traded at the cost of reduced achievable performance with respect to NMPC. Computational advantages of the convex optimization compared with nonlinear programming, required for NMPC, is commonly known and comparison studies are widely spread. Therefore a computational effort comparison is not a part of this paper.    
\begin{figure}[htbp]
\begin{center}
\includegraphics[width=0.48\textwidth]{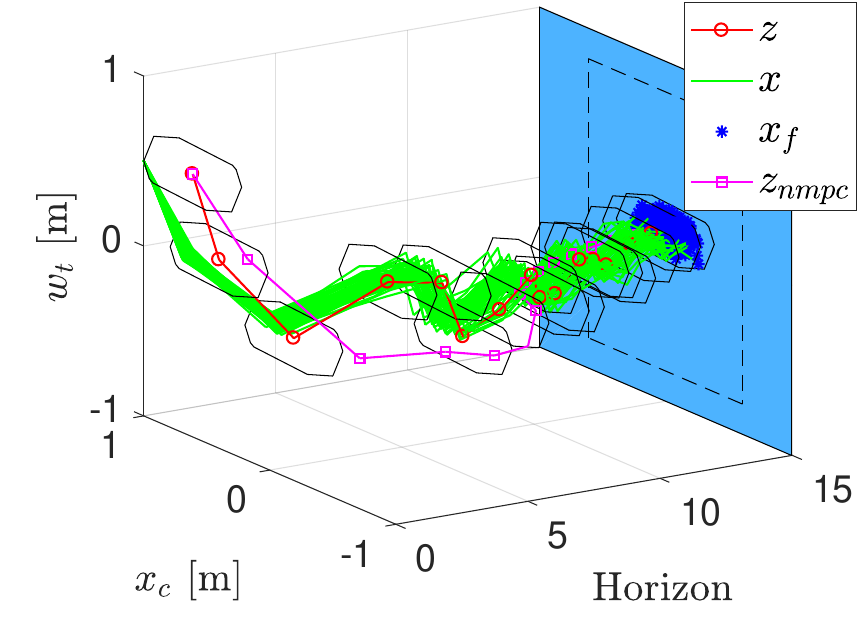}
\vspace{-3mm}
\caption{Projection of TMPC on $x_c$ and $\omega_t$ along the horizon} 
\label{fig:TMPC1_1}
\end{center}
\end{figure}
% 
%\vspace{-3mm}
%
\begin{figure}[htbp]
\begin{center}
\includegraphics[width=0.48\textwidth]{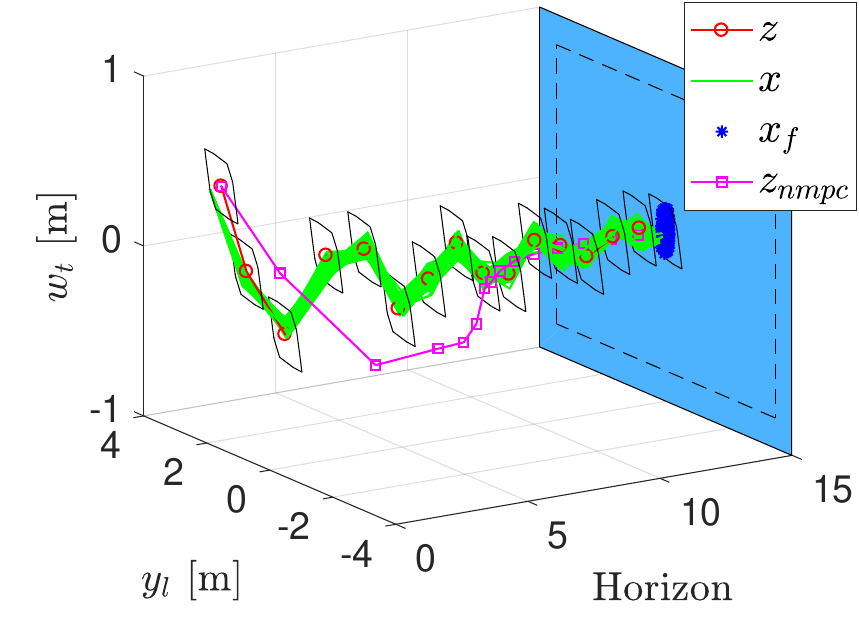}
\vspace{-3mm}
\caption{Projection of TMPC on $y_l$ and $\omega_t$ along the horizon} 
\label{fig:TMPC2_2}
\end{center}
\end{figure}
% 

%\vspace{-1mm}
\section{CONCLUSION}
\label{sec:conc}
%  Beneath an active oscillation damping of the first bending mode
%
%
In this paper, a tube-based model predictive control (TMPC) scheme for an active vibration damping control of stacker cranes (STC) is proposed. It offers a less computational demanding solution for both trajectory planning and tracking control. Additionally, mismatches due to the formulation of a nonlinear STC model as a linear parameter varying system are captured by the TMPC. In addition, a simple tube parametrization is introduced. For nominal trajectory planning the frame of soft-constrained MPC was used to penalize any close operation beside the resonance frequencies. The proposed scheme establishes the closed-loop stability and recursive feasibility, which are performed on a numerical example of STC and benchmarked with a nonlinear MPC.
%
%
%   
%
%\vspace{-1mm}
%\input{Kapitel/Bib} % thebibliography 
\bibliography{Chapter/LiteratureBib} %Bibtex
\bibliographystyle{IEEEtran}
%%%%%%%%%%%%%%%%%%%%%%%%%%%%%%%%%%%%%%%%%%%%%%%%%%%%%%%%%%%%%%%%%%%%%%%%%%%%%%%%

%\addtolength{\textheight}{-10cm}

%%%%%%%%%%%%%%%%%%%%%%%%%%%%%%%%%%%%%%%%%%%%%%%%%%%%%%%%%%%%%%%%%%%%%%%%%%%%%%%%

\end{document}